\documentclass{amsart}

\title[On cancellation for 4-manifolds]{Cancellation for 4-manifolds with\\ virtually abelian fundamental group}
\author{Qayum Khan}
\address{Department of Mathematics, Saint Louis University, St Louis MO 63103 U.S.A.}
\email{khanq@slu.edu}

\usepackage[leqno]{amsmath}
\usepackage{amssymb,amsthm,amscd}
\usepackage{txfonts,mathrsfs}

\input xy
\xyoption{all}

\usepackage{xcolor}
\definecolor{dark-red}{rgb}{0.4,0.15,0.15}
\definecolor{dark-blue}{rgb}{0.15,0.15,0.4}
\definecolor{medium-blue}{rgb}{0,0,0.5}
\usepackage[colorlinks, linkcolor={dark-red}, citecolor={dark-blue}, urlcolor={medium-blue}]{hyperref} 

\newtheorem{thm}{Theorem}[section]
\newtheorem{cor}[thm]{Corollary}
\newtheorem{prop}[thm]{Proposition}
\newtheorem{lem}[thm]{Lemma}

\theoremstyle{definition}
\newtheorem{defn}[thm]{Definition}

\theoremstyle{remark}
\newtheorem{rem}[thm]{Remark}

\newcommand{\nc}{\newcommand}
\newcommand{\qbold}[1]{\boldsymbol{\mathsf{#1}}}

\newcommand{\ncc}[2]{\newcommand{#1}{\mathcal{#2}}}
\ncc{\cS}{S}

\newcommand{\ncf}[2]{\newcommand{#1}{\mathscr{#2}}}
\ncf{\fH}{H}

\newcommand{\nck}[2]{\newcommand{#1}{\mathfrak{#2}}}
\nck{\km}{m}
\nck{\kM}{M}
\nck{\kP}{P}

\newcommand{\ncm}[2]{\newcommand{#1}{\mathrm{#2}}}
\ncm{\Center}{Center}
\ncm{\Closed}{ClosedSets}
\ncm{\closure}{closure}
\ncm{\DIFF}{DIFF}
\ncm{\Ker}{Ker}
\ncm{\Img}{Im}
\ncm{\intr}{int}
\ncm{\maxspec}{maxspec}
\ncm{\nbhd}{nbhd}
\ncm{\Nil}{Nil}
\ncm{\pt}{pt}
\ncm{\rad}{rad}
\ncm{\spec}{spec}
\ncm{\Split}{split}
\ncm{\TOP}{TOP}
\ncm{\UNil}{UNil}
\ncm{\Wh}{Wh}

\newcommand{\ncq}[2]{\newcommand{#1}{\qbold{#2}}}
\ncq{\C}{C}
\ncq{\CP}{CP}
\ncq{\D}{D}
\ncq{\Id}{1}
\nc{\N}{\qbold{Z}_{\geqslant 0}}
\ncq{\R}{R}
\ncq{\RP}{RP}
\ncq{\Z}{Z}

\nc{\gens}[1]{\left\langle #1 \right\rangle}
\nc{\Inn}[2]{\left\langle #1, #2 \right\rangle}
\nc{\ol}[1]{\overline{#1}}
\nc{\prn}[1]{\left( #1 \right)}
\nc{\set}[1]{\left\{\, #1 \,\right\}}
\nc{\wt}[1]{{\widetilde{#1}}}
\nc{\xra}[1]{\xrightarrow{#1}}

\nc{\bdry}{\partial}
\nc{\en}{\enspace}
\nc{\G}{\Gamma}
\nc{\homeo}{\approx}
\nc{\inv}{^{-1}}
\nc{\iso}{\cong}
\nc{\rel}{\;\mathrm{rel}\,}
\nc{\ST}{\;|\;}
\nc{\vphi}{\varphi}
\nc{\x}{\times}
\nc{\xo}{\otimes}

\begin{document}

\begin{abstract}
Suppose $X$ and $Y$ are compact connected topological 4-manifolds with fundamental group $\pi$.
For any $r \geqslant 0$, \textbf{$Y$ is $r$-stably homeomorphic to $X$} if $Y \# r(S^2 \x S^2)$ is homeomorphic to $X \# r(S^2\x S^2)$.
How close is stable homeomorphism to homeomorphism?

When the common fundamental group $\pi$ is virtually abelian, we show that large $r$ can be diminished to $n+2$, where $\pi$ has a finite-index subgroup that is free-abelian of rank $n$.
In particular, if $\pi$ is finite then $n=0$, 
hence $X$ and $Y$ are $2$-stably homeomorphic, which is one $S^2 \times S^2$ summand in excess of the cancellation theorem of Hambleton--Kreck \cite{HK}.

The last section is a case-study investigation of the homeomorphism classification of closed manifolds in the tangential homotopy type of $X = X_- \# X_+$, where $X_\pm$ are closed nonorientable topological 4-manifolds whose fundamental groups have order two \cite{HKT}.
\end{abstract}

\maketitle

\section{Introduction}

Suppose $X$ is a compact connected smooth 4-manifold, with fundamental group $\pi$ and orientation character $\omega : \pi \to \{\pm 1\}$.
Our motivation herein is the Cappell--Shaneson stable surgery sequence \cite[3.1]{CS1}, whose construction involves certain stable diffeomorphisms.
These explicit self-diffeomorphisms lead to a modified version of Wall realization $\rel \bdry X$:
\begin{equation}\label{eqn:CS}
L_5^s(\Z[\pi^\omega]) \x \cS_\DIFF^s(X) \xra{\qquad} \ol{\cS}_\DIFF^s(X),
\end{equation}
where $\cS$ is the simple smooth structure set and $\ol{\cS}$ and is the stable structure set.
Recall that the equivalence relation on these structure sets is smooth $s$-bordism of smooth manifold homotopy structures.
The actual statement of \cite[Theorem 3.1]{CS1} is sharper in that the amount of stabilization, that is, the number of connected summands of $S^2 \x S^2$, depends only on the rank of a representative of a given element of the odd-dimensional $L$-group.

In the case $X$ is \emph{sufficiently large}, in that it contains a two-sided incompressible smooth 3-submanifold $\Sigma$, a periodicity argument using Cappell's decomposition \cite[7]{CappellFree} shows that the restriction of the above action on $\cS_{\DIFF}^s(X)$ to the summand $\UNil_5^s$ of $L_5^s(\Z[\pi^\omega])$ is free.
Therefore for each nonzero element of this exotic $\UNil$-group, there exists a distinct, stable, smooth homotopy structure on $X$, restricting to a diffeomorphism on $\bdry X$, which is not $\Z[\pi_1(\Sigma)]$-homology splittable along $\Sigma$.
If $\Sigma$ is the 3-sphere, the $\TOP$ case is \cite{Khan_connect}.
Furthermore, when $X$ is a connected sum of two copies of $\RP^4$, see \cite{JK} and \cite{BDK}.

For any $r \geqslant 0$, denote the \textbf{$r$-stabilization of $X$} by
\[
X_r := X \# r (S^2 \x S^2).
\]

\subsection*{Acknowledgments}

I would like to thank Jim Davis for having interested me in relating stabilization to non-splittably fake connected sums of 4-manifolds.
Completed under his supervision, this long-delayed paper constitutes a chapter of the author's thesis \cite{Khan_Dissertation};
note Proposition~\ref{PropB} was recently extended from virtually cyclic to virtually abelian groups.

\section{On the topological classification of 4-manifolds}\label{Sec_Classification}

The main result (\ref{Cor_main}) of this section is an upper bound on the number of $S^2\x S^2$ connected summands sufficient for a \emph{stable homeomorphism,} where the fundamental group of $X$ lies in a certain class of good groups.
By using Freedman--Quinn surgery \cite[\S 11]{FQ}, if $X$ is also sufficiently large (\ref{Cor_PropB} for example), each nonzero element $\vartheta$ of the $\UNil$-group and simple $\DIFF$ homotopy structure $(Y,h: Y \to X)$ pair to form a distinct $\TOP$ homotopy structure $(Y_\vartheta,h_\vartheta)$ that represents the $\DIFF$ homotopy structure $\vartheta \cdot (Y,h)$ obtained from \eqref{eqn:CS}.

\subsection{Statement of results}

For finite groups $\pi$, the theorem's conclusion and the proof's topology are similar to Hambleton--Kreck \cite{HK}.
However, the algebra is quite different.

\begin{thm}\label{ThmB}
Suppose $\pi$ is a good group (in the sense of \cite{FQ}) with orientation character $\omega: \pi \to \{\pm 1\}$.
Consider $A := \Z[\pi^\omega]$, a group ring with involution: $\ol{g} = \omega(g) g^{-1}$.
Select an involution-invariant subring $R$ of the commutative $\Center(A)$.
Its norm subring is
\[
R_0 ~:=~ \set{\sum_i x_i \ol{x}_i \;\bigg|\; x_i \in R}.
\]
Suppose $A$ is a finitely generated $R_0$-module, $R_0$ is noetherian, and the dimension $d$ is finite:
\[
d := \dim(\maxspec\, R_0) < \infty.
\]

Now suppose that $X$ is a compact connected $\TOP$ 4-manifold with
\[
(\pi_1(X), w_1) = (\pi,\omega)
\]
and that it has the form
\[
(X,\bdry X) = (X_{-1}, \bdry X) \# (S^2 \x S^2).
\]
If $X_r$ is homeomorphic to $Y_r$ for some $r \geqslant 0$, then $X_d$ is homeomorphic to $Y_d$.
\end{thm}

Here are the class of examples of good fundamental groups promised in the paper's title.

\begin{prop}\label{PropB}
Suppose $\pi$ is a finitely generated, virtually abelian group, with any homomorphism $\omega: \pi \to \{\pm 1\}$.
For some $R$, the pair $(\pi,\omega)$ satisfies the above hypotheses:
$\pi$ is good, $A$ is a finitely generated $R_0$-module, $R_0$ is noetherian, and $d$ is finite.
Furthermore, $d=n+1$, where $\pi$ contains a finite-index subgroup that is free-abelian of finite rank $n \geqslant 0$.
\end{prop}

The author's original motivations are infinite virtually cyclic groups of the second kind.

\begin{cor}\label{Cor_PropB}
Let $X$ be a compact connected $\TOP$ 4-manifold whose fundamental group is an amalgamated product $G_- *_F G_+$ with $F$ a finite common subgroup of $G_\pm$ of index two.
If $Y$ is stably homeomorphic to $X$, then $Y \# 3(S^2\x S^2)$ is homeomorphic to $X \# 3(S^2\x S^2)$.
\end{cor}

\begin{proof}
Division of $G_\pm$ by the normal subgroup $F$ yields a short exact sequence of groups:
\[\begin{CD}
1 @>>> F @>>> G_- *_F G_+ @>>> \C_2 * \C_2 \iso \C_\infty \rtimes_{-1} \C_2 @>>> 1.
\end{CD}\]
So $\pi_1(X)$ contains an infinite cyclic group of finite index (namely, twice the order of $F$).
Now apply Proposition~\ref{PropB} with $n=1$ ($d=2$).
Then apply Theorem~\ref{ThmB} to $X \# (S^2 \times S^2)$.
\end{proof}

Given the full strength of the proposition, we generalize the above specialized corollary.

\begin{cor}\label{Cor_main}
Let $X$ be a compact connected $\TOP$ 4-manifold whose fundamental group is virtually abelian: say $\pi_1(X)$ contains a finite-index subgroup that is free-abelian of rank $n < \infty$.
If $Y$ is stably homeomorphic to $X$, then $Y$ is $(n+2)$-stably homeomorphic to $X$.
\qed
\end{cor}

More generally, can we reach the same conclusion if $\pi$ has a finite-index subgroup $\G$ that is \emph{polycyclic} of Hirsch length $n$?
The example $\pi = \Z^2 \rtimes_{\left(\begin{smallmatrix}2 & 1\\ 1 & 1\end{smallmatrix}\right)} \Z$ is not virtually abelian.


\subsection{Definitions and lemmas}

An exposition of the following concepts with applications is available in Bak's book \cite{BakBook}.
We assume the reader knows the more standard notions.

\begin{defn}[{\cite[I:4.1]{Bass2}}]
A \textbf{unitary ring} $(A, \lambda, \Lambda)$ consists of a ring with involution $A$, an element
\[
\lambda \in \Center(A) \quad\text{satisfying}\quad \lambda
\ol{\lambda} = 1,
\]
and a \textbf{form parameter} $\Lambda$.
This is an abelian subgroup of $A$ satisfying
\[
\set{a + \lambda \ol{a} \ST a \in A} \subseteq \Lambda \subseteq
\set{a \in A \ST a - \lambda \ol{a} = 0}
\]
and
\[
r a \ol{r} \in \Lambda \quad\text{for all}\quad r \in A \text{ and }
a \in \Lambda.
\]
\end{defn}

Here is a left-handed classical definition discussed in the equivalence after its reference.

\begin{defn}[{\cite[I:4.4]{Bass2}}]
We regard a \textbf{quadratic module} over a unitary ring $(A,\lambda,\Lambda)$ as a triple $(M,\Inn{\cdot}{\cdot}, \mu)$ consisting of a left $A$-module $M$, a bi-additive function
\[
\Inn{\cdot}{\cdot}: M \times M \xra{\quad} A \quad\text{such that}\quad \Inn{ax}{by} = a \, \Inn{x}{y} \, \ol{b} \quad\text{and}\quad \Inn{y}{x} = \lambda \, \ol{\Inn{x}{y}}
\]
(called a \textbf{$\lambda$-hermitian form}), and a function (called a \textbf{$\Lambda$-quadratic refinement})
\[
\mu: M \xra{\quad} A/\Lambda \quad\text{such that}\quad \mu(ax) = a \, \mu(x) \, \ol{a} \quad\text{and}\quad [\Inn{x}{y}] = \mu(x+y)-\mu(x)-\mu(y).
\]
\end{defn}

The following unitary automorphisms can be realized by diffeomorphisms \cite[1.5]{CS1}.

\begin{defn}[{\cite[I:5.1]{Bass2}}]\label{Def_transvection}
Let $(M,\Inn{\cdot}{\cdot}, \mu)$ be a quadratic module over a unitary ring $(A,\lambda,\Lambda)$. A \textbf{transvection} $\sigma_{u,a,v}$ is an isometry of this structure defined by the formula
\[
\sigma_{u,a,v}: M \xra{\quad} M; \qquad x \longmapsto x + \Inn{v}{x} u - \ol{\lambda} \Inn{u}{x} v - \ol{\lambda} \Inn{u}{x} a u
\]
where $u,v \in M$ and $a \in A$ are elements satisfying
\[
\Inn{u}{v} = 0 \in A \quad\text{and}\quad
\mu(u) = 0 \in A/\Lambda \quad\text{and}\quad
\mu(v) = [a] \in A/\Lambda.
\]
\end{defn}

The following lemmas involve, for any finitely generated projective $A$-module $P = P^{**}$, a nonsingular $(+1)$-quadratic form over $A$ called the \textbf{hyperbolic construction}
\[
\fH(P) := ( P \oplus P^*, \Inn{\cdot}{\cdot}, \mu ) \en\text{where}\en
\Inn{x + f}{y + g} := f(y) + \ol{g(x)} \en\text{and}\en
\mu(x + f) := [f(x)].
\]
Topologically, $\fH(A)$ is the equivariant intersection form of $S^2 \x S^2$ with coefficients in $A$.

\begin{lem}\label{LemB1}
Consider a compact connected $\TOP$ 4-manifold $X$ with good fundamental group $\pi$ and orientation character $\omega: \pi \to
\{\pm 1\}$.
Define a ring with involution $A := \Z[\pi^\omega]$.
Suppose that there is an orthogonal decomposition
\[
K := \Ker\,w_2(X) = V_0 \perp V_1
\]
as quadratic submodules of the intersection form of $X$ over $A$, with a nonsingular restriction to $V_0$.
Define a homology class and a free $A$-module
\begin{eqnarray*}
p_+ &:=& [S^2_+ \x \pt] \\
P_+ &:=& A p_+.
\end{eqnarray*}
Consider the summand
\[
\fH(P_+) = H_2(S^2_+ \x S^2; A)
\]
of
\[
H_2(X \# (S_+^2 \x S^2) \# (S_-^2 \x S^2);A).
\]
Then for any transvection $\sigma_{p,a,v}$ on the quadratic module $K \perp \fH(P_+)$ with $p \in V_0 \oplus P_+$ and $v \in K$, the
stabilized isometry $\sigma_{p,a,v} \oplus \Id_{H_2(2(S^2\x S^2);A)}$ can be realized by a self-homeomorphism of $X \# 3(S^2 \x S^2)$ which restricts to the identity on $\bdry X$.
\end{lem}

\begin{rem}
In the case that $\bdry X$ is empty and $\pi_1(X)$ is finite, then Lemma \ref{LemB1} is exactly \cite[Corollary 2.3]{HK}.
Although it turns out that their proof works in our generality, we include a full exposition, providing details absent from Hambleton--Kreck \cite{HK}.
\end{rem}

\begin{lem}\label{LemB1a}
Suppose $X$ and $p$ satisfy the hypotheses of Lemma \ref{LemB1}.
If $p$ is unimodular in $V_0 \oplus P_+$, then the summand $X_1 = X \# (S_+^2 \x S^2)$ of $X_2$ can be topologically re-split so that $S^2
\x \pt$ represents $p$.
\end{lem}

\begin{proof}
Since $V_0 \perp \fH(P_+)$ is nonsingular, there exists an element $q \in V_0 \perp \fH(P_+)$ such that $(p,q)$ is a hyperbolic pair.
Since $p,q \in \Ker\, w_2(X_1)$ and $w_2$ is the sole obstruction to framing the normal bundle in the universal cover, each homology
class is represented by a canonical regular homotopy class of framed immersion
\[
\alpha, \beta: S^2 \x \R^2 \xra{\quad} X_1
\]
with transverse double-points.
Since the self-intersection number of $\alpha$ vanishes, all its double-points pair to yield framed immersed Whitney discs; consider each disc separately:
\[
W: D^2 \x \R^2 \xra{\quad} X_1.
\]
Upon performing finger-moves to regularly homotope $W$, assume that one component of
\[
\alpha(S^2\x 0) \setminus W(\bdry D^2 \x \R^2)
\]
is a framed embedded disc
\[
V: D^2 \x \R^2 \xra{\quad} X_1
\]
and, by an arbitrarily small regular homotopy of $\beta$, that $\beta| S^2 \x 0$ is transverse to $W| \intr\, D^2 \x 0$ with
algebraic intersection number $1$ in $\Z[\pi_1(X_1)]$.
Hence $W$ is a framed properly immersed disc in
\[
\ol{X}_1 := X_1 \setminus \Img\,V.
\]
So, since $\pi_1(\ol{X}_1) \iso \pi_1(X)$ is a good group, by Freedman's disc theorem \cite[5.1A]{FQ}, there exists a framed properly $\TOP$ embedded disc
\[
W': D^2 \x \R^2 \xra{\quad} \ol{X}_1
\]
such that
\[
W' = W \text{ on } \bdry D^2 \x \R^2 \quad\text{and}\quad \Img\,W'
\subset \Img\,W.
\]
Therefore, by performing a Whitney move along $W'$, we obtain that $\alpha$ is regularly homotopic to a framed immersion with one fewer
pair of self-intersection points.
Thus $\alpha$ is regularly homotopic to a framed $\TOP$ embedding $\alpha'$.
A similar argument, allowing an arbitrarily small regular homotopy of $\alpha'$, shows that $\beta$ is regularly homotopic to a framed $\TOP$ embedding $\beta'$ transverse to $\alpha'$, with a single intersection point
\[
\alpha'(x_0\x 0) = \beta'(y_0\x 0)
\]
such that the open disc
\[
\Delta := \beta'(y_0\x \R^2) \subset\alpha'(S^2\x 0).
\]
Define a closed disc
\[
\Delta' := S^2 \setminus (\alpha')\inv(\Delta).
\]
Surgery on $X_1$ along $\beta'$ yields a compact connected $\TOP$ 4-manifold $X'$. Hence $X_1$ is recovered by surgery on $X'$ along the framed embedded circle
\[
\gamma: S^1 \x \R^3 \homeo \nbhd_{S^2}(\bdry\Delta') \x \R^2 \xra{\en\alpha'\en} X_1 \setminus \Img\,\beta' \subset X'.
\]
But the circle $\gamma$ is trivial in $X'$, since it extends via $\alpha'$ to a framed embedding of the disc $\Delta'$ in $X'$.
Therefore we obtain a $\TOP$ re-splitting of the connected sum
\[
X_1 \homeo X' \# (S^2 \x S^2)
\]
so that $S^2\x \pt$ of the right-hand side represents the image of $p$.
\end{proof}

The next algebraic lemma decomposes certain transvections so that the pieces fit into the previous topological lemma.

\begin{lem}\label{LemB1b}
Suppose $(A,\lambda,\Lambda)$ is a unitary ring such that: the additive monoid of $A$ is generated by a subset $S$ of the unit
group $(A^\times, \cdot)$.
Let $K = V_0 \perp V_1$ be a quadratic module over $(A,\lambda,\Lambda)$ with a nonsingular restriction to
$V_0$, and let $P_\pm$ be free left $A$-modules of rank one.
Then any stabilized transvection
\[
\sigma_{p,a,v} \oplus \Id_{\fH(P_-)} \quad\text{on}\quad K\perp
\fH(P_+)\perp \fH(P_-)
\]
with $p \in V_0 \oplus P_+$ and $v \in K$ is a composite of transvections $\sigma_{p_i,0,v_j}$ with unimodular $p_i \in V_0 \oplus P_+$ and isotropic $v_j \in K \oplus \fH(P_-)$.
\end{lem}

\begin{proof}
Using a symplectic basis $\set{p_\pm,q_\pm}$ of each hyperbolic plane $\fH(P_\pm)$, define elements of $K\oplus \fH(P_+ \oplus P_-)$:
\begin{eqnarray*}
v_0 &:=& v + p_- - a q_- \\ v_1 &:=& -p_- \\ v_2 &:=& a q_-.
\end{eqnarray*}
Then
\[
v = \sum_{i=0}^2 v_i.
\]
Observe that each $v_i \in K \oplus \fH(P_-)$ is isotropic with $\Inn{v_i}{p} = 0$.
So transvections $\sigma_{p,0,v_j}$ are defined.
Note, by Definition \ref{Def_transvection}, for all $x \in K \oplus \fH(P_+ \oplus P_-)$, that
\begin{eqnarray*}
(\sigma_{p,0,v_2} \circ \sigma_{p,0,v_1} \circ \sigma_{p,0,v_0})(x)
&=& x + \sum_i \Inn{v_i}{x} p - \sum_i \ol{\lambda} \Inn{p}{x} v_i - \ol{\lambda} \Inn{p}{x} \sum_{i<j} \Inn{v_j}{v_i} p\\
&=& x + \Inn{v}{x} p - \ol{\lambda} \Inn{p}{x} v - \ol{\lambda} \Inn{p}{x} a p \\
&=& \sigma_{p,a,v}(x) \oplus \Id_{\fH(P_-)}.
\end{eqnarray*}

Therefore it suffices to consider the case that $v \in K \oplus \fH(P_-)$ is isotropic.
Write
\[
p = p' \oplus p'' \in V_0 \oplus P_+.
\]
Define a unimodular element
\[
p_0 := p' \oplus 1 p_+.
\]
Note, since $P_+$ has rank one and by hypothesis, there exist $n \in
\N$ and unimodular elements $p_1, \ldots, p_n \in S p_+ \subseteq
P_+$ such that
\[
p - p_0 = p''- 1 p_+ = \sum_{i=1}^n p_i.
\]
For each $1 \leqslant i \leqslant n$, write
\[
p_i := s_i p_+ \quad\text{for some}\quad s_i \in S.
\]
Observe for all $1 \leqslant i,j \leqslant n$ that
\begin{gather*}
\Inn{v}{p_i} = 0 \\
\mu(p_i) ~=~ s_i \mu(p_+) \ol{s_i} ~=~ 0 \\
\Inn{p_i}{p_j} ~=~ s_i \Inn{p_+}{p_+} \ol{s_j} ~=~ 0.
\end{gather*}
Hence, we also have
\begin{eqnarray*}
\Inn{v}{p_0} &=& 0 \\ \mu(p_0) &=& 0.
\end{eqnarray*}
Then transvections $\sigma_{p_i,0,v}$ are defined and commute, so note
\[
\sigma_{p,0,v} = \prod_{i=0}^n \sigma_{p_i,0,v}.\qedhere
\]
\end{proof}

\begin{proof}[Proof of Lemma \ref{LemB1}.]
Define a homology class and a free $A$-module
\begin{eqnarray*}
p_- &:=& [S_-^2 \x \pt] \\ P_- &:=& A p_-.
\end{eqnarray*}
Consider the $A$-module decomposition
\[
H_2(X_2; A) = H_2(X;A) \oplus \fH(P_+) \oplus \fH(P_-).
\]
Observe that the unitary ring
\[
(A,\lambda,\Lambda) = (\Z[\pi^\omega], +1, \set{a - \ol{a} \ST a \in
A})
\]
satisfies the hypothesis of Lemma \ref{LemB1b} with the multiplicative subset
\[
S = \pi \cup -\pi.
\]
Therefore the stabilized transvection
\[
\sigma_{p,a,v}\oplus \Id_{\fH(P_-)}
\]
is a composite of transvections $\sigma_{p_i,0,v_j}$ with unimodular $p_i \in V_0 \oplus P_+$ and isotropic $v_j \in K \oplus \fH(P_-)$.
Then by Lemma \ref{LemB1a}, for each $i$, a $\TOP$ re-splitting
\[
f_i: X_1 \homeo X' \# (S^2 \x S^2)
\]
of the connected sum can be chosen so that $S^2 \x \pt$ represents $p_i$.
So by the Cappell--Shaneson realization theorem \cite[1.5]{CS1}\footnote{Their theorem realizes any
transvection of the form $\sigma_{p_+,a,v}$ by a diffeomorphism of the $1$-stabilization.}, for each $i$ and $j$, the pullback under $(f_i)_*$ of the stabilized transvection
\[
\sigma_{p_i,0,v_j} \oplus \Id_{H_2(S^2\x S^2;A)} = \sigma_{p_i\oplus
0, 0, v_j \oplus 0}
\]
is an isometry induced by a self-diffeomorphism of \[ (X' \# (S^2\x
S^2)) \# (S^2 \x S^2).
\]
Hence, by conjugation with the homeomorphism $f_i$, the above isometry is induced by a self-homeomorphism of
\[
X_2 = X_1 \# (S^2 \x S^2).
\]
Thus the stabilized transvection
\[
(\sigma_{p,a,v}\oplus \Id_{\fH(P_-)}) \oplus \Id_{H_2(S^2\x S^2; A)}
\]
is induced by the stabilized composite self-homeomorphism of
\[
X_3 = X_2 \# (S^2 \x S^2).\qedhere
\]
\end{proof}

\subsection{Proof of the main theorem}

Now we modify the induction of \cite[Proof B]{HK}; our result will be one $S^2 \x S^2$ connected summand less efficient than
Hambleton--Kreck \cite{HK} in the case that $\pi$ is finite.
The main algebraic technique is a theorem of Bass \cite[IV:3.4]{Bass2} on the transitivity of a certain subgroup of
isometries on the set of hyperbolic planes.
We refer the reader to \cite[\S IV:3]{Bass2} for the terminology used in our proof.
The main topological technique is a certain clutching construction of an $s$-cobordism.

\begin{proof}[Proof of Theorem \ref{ThmB}]
We may assume $r \geqslant d+1$. Let
\[
f: X \# r(S^2\x S^2) \xra{\quad} Y \# r(S^2\x S^2)
\]
be a homeomorphism.
We show that
\[
\ol{X} := X \# (r-1)(S^2\x S^2)
\]
is homeomorphic to
\[
\ol{Y} := Y \# (r-1)(S^2\x S^2),
\]
thus the result follows by backwards induction on $r$.

Consider Definition \ref{Def_quasifinite} and \cite[Hypotheses IV:3.1]{Bass2}.
By our hypothesis and Lemma \ref{LemB2}, the minimal form parameter
\[
\Lambda := \set{a - \ol{a}\ST a \in A}
\]
makes $(A,\Lambda)$ a quasi-finite unitary $(R,+1)$-algebra.
Note, since
\[
X = X_{-1} \# ((S')^2 \x S^2)
\]
by hypothesis, that the rank $r+1$ free $A$-module summand
\[
P := H_2\prn{ (S')^2 \x \pt \sqcup r (S^2 \x \pt); A }
\]
of
\[
\Ker\,w_2(\ol{X}\#(S^2\x S^2))
\]
satisfies \cite[Case IV:3.2(a)]{Bass2}.
Then, by \cite[Theorem IV:3.4]{Bass2}, the subgroup $G$ of the group $U(\fH(P))$ of unitary automorphisms defined by
\[
G := \gens{\; \fH(E(P)),\; EU(\fH(P))\; }
\]
acts transitively on the set of hyperbolic pairs in $\fH(P)$.
So, by \cite[Corollary IV:3.5]{Bass2} applied to the quadratic module
\[
V := \Ker\, w_2(X_{-1}),
\]
the subgroup $G_1$ of $U(V \perp \fH(P))$ defined by
\[
G_1 := \gens{\; \Id_V \perp G,\; EU(\fH(P),P;V),\;
EU(\fH(P),\ol{P};V)\; }
\]
acts transitively on the set of hyperbolic pairs in $V\perp \fH(P)$.
Let
\[
(p_0,q_0) \quad\text{and}\quad (p_0', q_0')
\]
be the standard basis of the summand $H_2(S^2\x S^2; A)$ of
\[
H_2(\ol{X}\# (S^2 \x S^2); A) \quad\text{and}\quad H_2(\ol{Y}\# (S^2
\x S^2); A).
\]
Therefore there exists an isometry $\vphi \in G_1$ of
\[
V\perp \fH(P) = \Ker\, w_2(\ol{X} \# (S^2\x S^2))
\]
such that
\[
\vphi(p_0,q_0) = (f_*)\inv(p_0',q_0').
\]

\begin{lem}\label{LemB3}
The isometry
\[
\vphi\oplus \Id_{H_2(3(S^2\x S^2);A)}
\]
is induced by a self-homeomorphism $g$ of
\[
\ol{X} \# 4 (S^2 \x S^2).
\]
\end{lem}

Then the homeomorphism
\[
h := (f \# \Id_{3(S^2\x S^2)}) \circ g : \ol{X} \# 4(S^2 \x S^2)
\xra{\quad} \ol{Y} \# 4(S^2 \x S^2)
\]
satisfies the equation
\[
h_*(p_i,q_i) = (p_i', q_i') \quad\text{for all}\quad 0 \leqslant i \leqslant
3.
\]
Here the hyperbolic pairs
\[
\set{(p_i,q_i)}_{i=1}^3 \quad\text{and}\quad
\set{(p_i',q_i')}_{i=1}^3
\]
in the last three $S^2 \x S^2$ summands are defined similarly to
$(p_0,q_0)$ and $(p_0',q_0')$.

\begin{lem}\label{LemB4}
The manifold triad $(W;\ol{X},\ol{Y})$ is a compact $\TOP$ $s$-cobordism $\!\! \rel \bdry X$:
\[
W^5 := \ol{X} \x [0,1 ]\:\natural\: 4 (S^2 \x D^3) \; \bigcup_h\;
\ol{Y} \x [0,1 ] \:\natural\: 4 (S^2 \x D^3).
\]
\end{lem}

Therefore, since $\pi_1(\ol{X}) \iso \pi_1(X)$ is a good group, by the $\TOP$ $s$-cobordism theorem \cite[7.1A]{FQ},  $\ol{X}$ is homeomorphic to $\ol{Y}$.
This proves the theorem by induction on $r$.
\end{proof}

\begin{rem}
The reason for restriction to the $A$-submodule
\[
K = \Ker\, w_2(\ol{X} \# (S^2\x S^2))
\]
is two-fold.
Geometrically \cite[p504]{CS1}, a unique quadratic refinement of the intersection form exists on $K$, hence $K$ is maximal.
Also, the inverse image of $(p_0', q_0')$ under the isometry $f_*$ is guaranteed to be a hyperbolic pair in $K$, hence $K$ is simultaneously minimal.
\end{rem}

\subsection{Remaining lemmas and proofs}

\begin{defn}[{\cite[IV:1.3]{Bass2}}]\label{Def_quasifinite}
An $R_0$-algebra $A$ is \textbf{quasi-finite} if, for each maximal ideal $\km \in \maxspec(R_0)$, the following containment holds:
\[
\km A_\km \subseteq \rad\, A_\km
\]
and that the following ring is left artinian:
\[
A[\km] := A_\km / \rad\, A_\km.
\]
Here
\[
A_\km := (R_0)_\km \xo_{R_0} A
\]
is the localization of $A$ at $\km$, and $\rad\, A_\km$ is its Jacobson radical.
The pair $(A,\Lambda)$ is a \textbf{quasi-finite unitary $(R,\lambda)$-algebra} if $(A, \lambda, \Lambda)$ is a unitary ring, $A$ is an $R$-algebra with involution, and $A$ is a quasi-finite $R_0$-algebra.
Here $R_0$ is the subring of $R$ generated by \textbf{norms}:
\[
R_0 = \set{ \sum_i r_i \ol{r}_i \;\bigg|\; r_i \in R}.
\]
\end{defn}

\begin{lem}\label{LemB2}
Suppose $A$ is an algebra over a ring $R_0$ such that $A$ is a finitely generated left $R_0$-module.
Then $A$ is a quasi-finite $R_0$-algebra.
\end{lem}

\begin{proof}
Let $\km \in \maxspec(R_0)$.
By \cite[Corollary III:2.5]{Bass1} to Nakayama's lemma,
\[
A_\km \cdot \km = A_\km \cdot \rad\, (R_0)_\km \subseteq \rad\,A_\km.
\]
Then
\[
A[\km] = (A_\km / \km A_\km) / \prn{(\rad\, A_\km) / \km A_\km}
\]
and is a finitely generated module over the field
\[
(R_0)_\km / \km (R_0)_\km,
\]
by hypothesis.
Therefore $A[\km]$ is left artinian, hence $A$ is quasi-finite.
\end{proof}

The existence of the realization $g$ is proven algebraically; refer to \cite[\S II:3]{Bass2}.

\begin{proof}[Proof of Lemma \ref{LemB3}.]
Consider Lemma \ref{LemB1} applied to
\[
\ol{X} \#(S^2\x S^2) \qquad V_0 = \fH(P) \qquad V_1 = V.
\]
It suffices to show that the group $G_1$ is generated by a subset of the transvections $\sigma_{p,a,v}$ with $p \in \fH(P)$ and $v \in
V\oplus \fH(P)$.

By \cite[Cases II:3.10(1--2)]{Bass2}, the group
\[
EU(\fH(P))
\]
is generated by all transvections $\sigma_{u,a,v}$ with $u,v \in \ol{P}$ or $u,v \in P$.
By \cite[Case II:3.10(3)]{Bass2}, the group
\[
\fH(E(P))
\]
is generated by a subset of the transvections $\sigma_{u,a,v}$ with $u \in P, v \in \ol{P}$ or $u \in \ol{P}, v \in P$.
By \cite[Definition 1.4]{HK}, the group
\[
EU(\fH(P),P;V)
\]
is generated by all transvections $\sigma_{u,a,v}$ with $u \in P, v \in V$, and the group
\[
EU(\fH(P),\ol{P};V)
\]
is generated by all $\sigma_{u,a,v}$ with $u \in \ol{P}, v \in V$.
In any case, $p \in \fH(P)$ and $v \in V \oplus \fH(P)$.
\end{proof}

The assertion is essentially that $(W;\ol{X},\ol{Y})$ is a $h$-cobordism with zero Whitehead torsion.

\begin{proof}[Proof of Lemma \ref{LemB4}.]
By the Seifert--vanKampen theorem, we have a pushout diagram
\[\xymatrix{
\pi_1\prn{\ol{X} \x 1 \:\#\: 4(S^2\x S^2)} \ar[r]^{h_*}_{\iso}
\ar[d]_{\Id}
& \pi_1\prn{\ol{Y} \x [0,1] \:\natural\: 4(S^2 \x D^3)} \ar[d] \\
\pi_1\prn{\ol{X} \x [0,1] \:\natural\: 4(S^2\x D^3) } \ar[r] &
\pi_1(W). }
\]
So the maps induced by the inclusion $\ol{X} \sqcup \ol{Y} \to W$ are isomorphisms:
\begin{eqnarray*}
i_* &:& \pi_1(\ol{X} \x 0) \xra{\quad} \pi_1(W) \\
j_* &:& \pi_1(\ol{Y} \x 0) \xra{\quad} \pi_1(W).
\end{eqnarray*}
Denote $\pi$ as the common fundamental group using these identifications.

Observe that the nontrivial boundary map $\bdry_3$ of the cellular chain complex
\[
C_*(j; \Z[\pi])\; : \; 0 \xra{\quad}  \bigoplus_{0 \leqslant k < 4}
\Z[\pi] \cdot (S^2 \x D^3) \xra{\en h_\# \circ \bdry \en}
\bigoplus_{0 \leqslant l < 4} \Z[\pi] \cdot (D^2 \x S^2) \xra{\quad} 0
\]
is obtained as follows.
First, attach thickened $2$-cells to kill $4$ copies of the trivial circle in $\ol{Y}$.
Then, onto the resultant manifold
\[
\ol{Y} \:\#\: 4(S^2\x S^2),
\]
attach thickened $3$-cells to kill certain belt 2-spheres, which are the images under $h$ of the normal $2$-spheres to the $4$ copies of the trivial circle in $\ol{X}$.
Hence, as morphisms of based left $\Z[\pi]$-modules, the boundary map
\[
\bdry_3 = h_\# \circ \bdry
\]
is canonically identified with the morphism
\[
h_* = \Id : H_2(4(S^2\x S^2); \Z[\pi]) \xra{\quad} H_2(4(S^2\x S^2);
\Z[\pi])
\]
on homology induced by the attaching map $h$.
This last equality holds by the construction of $h$, since
\[
h_*(p_i,q_i) = (p_i', q_i') \quad\text{for all}\quad 0 \leqslant i < 4.
\]
So the inclusion $j: \ol{Y} \to W$ has torsion
\[
\tau(C_*(j; \Z[\pi])) ~=~ [h_\#] ~=~ [h_*] ~=~ [\Id] ~=~ 0 \in \Wh(\pi).
\]
A similar argument using $h\inv$ shows that the inclusion $i: \ol{X} \to W$ has zero torsion in $\Wh(\pi)$.
Therefore $(W;\ol{X},\ol{Y})$ is a compact $\TOP$ $s$-cobordism $\!\!\rel \partial X$.
\end{proof}

The final  proof of this section employs the theory of commutative rings and subrings (including invariant theory), as well as language from algebraic geometry ($\spec$ and $\maxspec$).

\begin{proof}[Proof of Proposition \ref{PropB}.]
Since $\pi$ is virtually polycyclic, it is a good group \cite[5.1A]{FQ}.
Since $\pi$ is virtually abelian, by intersection with finitely many conjugates of a finite-index abelian subgroup, we find an exact sequence of groups with $\G$ normal abelian and $G$ finite:
\[\begin{CD}
1 @>>> \G @>>> \pi @>>> G @>>> 1.
\end{CD}\]
This induces an action $G \curvearrowright \G$.
Consider these rings with involution and norm subring $R_0$:
\begin{eqnarray*}
A &:=& \Z[\pi^\omega] \quad\text{where}\en \forall g \in \pi: \ol{g} = \omega(g) g\inv\\
A_0 &:=& \Z[\G^\omega] \quad\text{which is a commutative ring}\\
R &:=& (A_0)^G ~=~ \left\{x \in A_0 \;|\; \forall g \in G : gx=x \right\}\\
R_0 &:=& \set{\textstyle \sum_i x_i \ol{x}_i \;\big|\; x_i \in R}.
\end{eqnarray*}
Note $\Z[\G^G] \subseteq R \subseteq \Center(A)$.
Since $\pi$ is finitely generated, by Schreier's lemma, so is $\G$.
So, by enlarging $G$ as needed, we may assume $\G$ is a free-abelian group of a finite rank $n$.

First, we show that $A$ is a finitely generated $R_0$-module.
Since $\G$ has only finitely many right cosets in $\pi$, the group ring $A$ is a finitely generated $A_0$-module.
Since $G$ is finite and $A_0$ is a finitely generated commutative ring and $\Z$ is a noetherian ring, by Bourbaki \cite[\S V.1: Theorem 9.2]{Bourbaki}, $A_0$ is a finitely generated $R$-module and $R$ is a finitely generated ring.
By Bass \cite[Intro IV:1.1]{Bass2}, the commutative ring $R$ is integral over its norm subring $R_0$.
So, since $R$ is a finitely generated integral $R_0$-algebra, it follows that $R$ is a finitely generated $R_0$-module \cite[Corollary~4.5]{Eisenbud}.
Therefore, $A$ is a finitely generated $R_0$-module.

Second, we show that $R_0$ is a noetherian ring.
It follows from Hilbert's basis theorem \cite[Corollary 1.3]{Eisenbud} that the finitely generated commutative $\Z$-alegebra $A_0$ is noetherian.
So $R_0$ is too, by Eakin's theorem \cite[A3.7a]{Eisenbud}, since $A_0$ is a finitely generated $R_0$-module.

Third, we show that the irreducible-dimension of the Zariski topology on $\kP := \spec(R_0)$ is $n+1$.
Here, by \textbf{irreducible-dimension} of a topological space, we mean the supremum of the lengths of proper chains of closed irreducible subsets, where \textbf{reducible} means being the union of two nonempty closed proper subsets \cite[\S I:1]{Hartshorne}.
Krull dimension of a ring equals irreducible-dimension of its $\spec$ \cite[II:3.2.7]{Hartshorne}; in particular $\dim(\kP)=\dim(R_0)$.
Since $A_0$ is a finitely generated $R_0$-module, $A_0$ is integral over $R_0$  by \cite[Corollary~4.5]{Eisenbud}.
So $\dim(R_0)=\dim(A_0)$, by the Cohen--Seidenberg theorems \cite[4.15, 4.18; Axiom D3]{Eisenbud}.
Note $\dim(A_0)=n+1$ since $\dim(\Z)=1$, by \cite[Exercise~10.1]{Eisenbud}.
Thus $\dim(\kP)=n+1$.

Last, we show the topological space $\kP = \spec(R_0)$ and its subspace $\kM := \maxspec(R_0)$ have equal irreducible-dimensions.
Since $R$ is a finitely generated commutative ring and $R$ is a finitely generated $R_0$-module, by the Artin--Tate lemma \cite[Exercise~4.32]{Eisenbud}, also $R_0$ is a finitely generated ring.
Then $R_0$ is a Jacobson ring, by the generalized Nullstellensatz \cite[Theorem~4.19]{Eisenbud}, since $\Z$ is Jacobson.
So we obtain an isomorphism of posets:
\[
\Closed(\kP) \xra{\quad} \Closed(\kM);\quad C \longmapsto C \cap \kM \quad\text{with inverse}\quad D \longmapsto \closure_{\kP}(D).
\]
This correspondence is worked out by Grothendieck \cite[\S IV.10: Proposition 1.2(c'); D\'efinitions 1.3, 3.1, 4.1; Corollaire 4.6]{EGA}. Hence $\dim(\kM) = \dim(\kP)$.
Thus $d = n+1$.
\end{proof}

\section{Manifolds in the tangential homotopy type of $\RP^4 \# \RP^4$}
\label{Sec_HomotopyType}

Given a tangential homotopy equivalence to a certain $\TOP$ 4-manifold, the main goal of this section is to uniformly quantify the amount of topological stabilization sufficient for smoothing and for splitting along a two-sided 3-sphere.
In particular, we sharpen a result of Jahren--Kwasik \cite[Theorem 1(f)]{JK} on connected sum of real projective 4-spaces (\ref{Cor_FakeBound}).


Let $X$ be a compact connected $\DIFF$ 4-manifold, and write
\[
(\pi,\omega) := (\pi_1(X),w_1(X)).
\]
Suppose $\pi$ is good \cite{FQ}. Let $\vartheta \in L_5^s(\Z[\pi^\omega])$; represent it by a simple unitary automorphism of the orthogonal sum of $r$ copies of the hyperbolic plane for some $r \geqslant 0$.
Recall \cite[\S 11]{FQ} that there exists a unique homeomorphism class
\[
(X_\vartheta,h_\vartheta) \in \cS_\TOP^s(X)
\]
as follows.
It consists of a compact $\TOP$ 4-manifold $X_\vartheta$ and a simple homotopy equivalence $h_\vartheta: X_\vartheta \to X$ that restricts to a homeomorphism $h: \bdry X_\vartheta \to \bdry X$ on the boundary, such that there exists a normal bordism rel $\bdry X$ from $h_\vartheta$ to $\Id_X$ with surgery obstruction
$\vartheta$.
Such a homotopy equivalence is called \emph{tangential;} equivalently, a homotopy equivalence $h: M \to X$ of $\TOP$ manifolds is \textbf{tangential} if the pullback microbundle $h^*(\tau_X)$ is isomorphic to $\tau_M$.

\begin{thm}\label{Thm_Real4}
The following $r$-stabilization admits a $\DIFF$ structure:
\[
X_\vartheta \# r(S^2\x S^2).
\]
Furthermore, there exists a $\TOP$ normal bordism between $h_\vartheta$ and $\Id_X$ with surgery obstruction $\vartheta \in
L_5^s(\Z[\pi^\omega])$, such that it consists of exactly $2r$ many 2-handles and $2r$ many 3-handles. In particular $X_\vartheta$ is $2r$-stably homeomorphic to $X$.
\end{thm}

\begin{proof}
The existence and uniqueness of $(X_\vartheta,h_\vartheta)$ follow from \cite[Theorems 11.3A, 11.1A, 7.1A]{FQ}.
But by \cite[Theorem 3.1]{CS1}, there exists a $\DIFF$ $s$-bordism class of $(X_\alpha,h_\alpha)$ uniquely determined as follows.
Given a rank $r$ representative $\alpha$ of the isometry class $\vartheta$, this pair $(X_\alpha,h_\alpha)$ consists of a compact $\DIFF$ 4-manifold
$X_\alpha$ and a simple homotopy equivalence $h_\alpha$ that restricts to a diffeomorphism on the boundary:
\begin{eqnarray*}
h_\alpha &:& (X_\alpha,\bdry X_\alpha) \xra{\quad} (X_r,\bdry X) \\
X_r &:=& X \# r(S^2\x S^2).
\end{eqnarray*}
It is obtained from a $\DIFF$ normal bordism $(W_\alpha, H_\alpha)$ rel $\bdry X$ from $h_\alpha$ to $\Id_{X_r}$ with of surgery obstruction $\vartheta$, constructed with exactly $r$ 2-handles and $r$ 3-handles, and clutched along a
diffeomorphism which induces the simple unitary automorphism $\alpha$ on the surgery kernel
\[
K_2(W_\alpha) = \fH\prn{\bigoplus_r \Z[\pi]}.
\]
This is rather the consequence, and not the construction\footnote{In the $\DIFF$ 4-dimensional case, via a self-diffeomorphism $\vphi$ inducing $\alpha$, embeddings are chosen within certain regular homotopy class of framed immersions of $2$-spheres. Cappell and Shaneson \cite[1.5]{CS1} cleverly construct $\vphi$ using a circle isotopy theorem of Whitney.} itself, of Wall realization \cite[6.5]{Wall} in high odd dimensions.

By uniqueness in the simple $\TOP$ structure set, the simple homotopy equivalences $h_\vartheta \# \Id_{r(S^2\x S^2)}$ and $h_\alpha$ are $s$-bordant.
Hence they differ by pre-composition with a homeomorphism, by the $s$-cobordism theorem \cite[Thm. 7.1A]{FQ}.
In particular, the domain $X_\vartheta \# r(S^2\x S^2)$ is homeomorphic to $X_\alpha$, inheriting its $\DIFF$ structure.
Therefore, post-composition of $H_\alpha$ with the collapse map $X_r \to X$ yields a normal bordism between the simple homotopy equivalences
$h_\vartheta$ and $\Id_X$, obtained by attaching $r+r$ 2- and 3-handles.
\end{proof}

Next, we recall Hambleton--Kreck--Teichner classification of the homemomorphism types and simple homotopy types of closed 4-manifolds with fundamental group $\C_2^-$.
Then, we shall give a partial classification of the simple homotopy types and stable homeomorphism types of their connected sums, which have fundamental group
$\D_\infty^{-,-} = \C_2^- * \C_2^-$.
The \emph{star operation} $*$ \cite[\S 10.4]{FQ} flips the Kirby--Siebenmann invariant of some 4-manifolds.

\begin{thm}[{\cite[Theorem 3]{HKT}}]
Every closed nonorientable topological 4-manifold with fundamental group order two is homeomorphic to exactly one manifold in the following list of so-called $w_2$-types.

\begin{enumerate}
\item[(I)]
The connected sum of $*\CP^2$ with $\RP^4$ or its star.
The connected sum of $k \geqslant 1$ copies of $\CP^2$ with $\RP^4$ or $\RP^2\x S^2$ or their stars.

\item[(II)]
The connected sum of $k \geqslant 0$ copies of $S^2 \x S^2$ with $\RP^2\x S^2$ or its star.

\item[(III)]
The connected sum of $k \geqslant 0$ copies of $S^2 \x S^2$ with $S(\gamma^1 \oplus\gamma^1 \oplus \varepsilon^1)$ or $\#_{S^1} r \RP^4$ or their stars, for unique $1 \leqslant r \leqslant 4$.
\end{enumerate}
\end{thm}

We explain the terms in the above theorem.
Firstly,
\[
\R \xra{\quad} \gamma^1 \xra{\quad} \RP^2
\]
denotes the canonical line bundle, and
\[
\varepsilon^1 := \R \x \RP^2
\]
denotes the trivial line bundle.
Secondly,
\[
S^2 \xra{\quad} S(\gamma^1 \oplus\gamma^1 \oplus \varepsilon^1) \xra{\quad} \RP^2
\]
is the sphere bundle of the Whitney sum.
Finally, the \textbf{circular sum}
\[
M \#_{S^1} N := M \setminus E \;\bigcup_{\bdry E}\; N \setminus E
\]
is defined by codimension zero embeddings of $E$ in $M$ and $N$ that are not null-homotopic, where $E$ is the nontrivial bundle:
\[
D^3 \xra{\quad} E \xra{\quad} S^1.
\]

\begin{cor}[{\cite[Corollary 1]{HKT}}]
Let $M$ and $M'$ be closed nonorientable topological 4-manifolds with fundamental group of order two. Then $M$ and $M'$ are (simple) homotopy equivalent if and only if

\begin{enumerate}
\item
$M$ and $M'$ have the same $w_2$-type,

\item
$M$ and $M'$ have the same Euler characteristic, and

\item
$M$ and $M'$ have the same Stiefel--Whitney number: $w_1^4[M] =
w_1^4[M']$ mod 2;

\item
$M$ and $M'$ have $\pm$ the same Brown--Arf invariant mod 8, in case of $w_2$-type III.
\end{enumerate}
\end{cor}

The following theorem is the main focus of this section.
The pieces $M$ and $M'$ are classified by Hambleton--Kreck--Teichner \cite{HKT}, and the $\UNil$-group is computed by Connolly--Davis \cite{CD}.
Since $\Z$ is a regular coherent ring, by Waldhausen's vanishing theorem \cite[Theorems 1,2,4]{Waldhausen}, $\wt{\Nil}_0(\Z;\Z^-,\Z^-) = 0$.
Hence $\UNil_5^s = \UNil_5^h$ \cite{CappellFree}.

\begin{thm}\label{Thm_TangentialBordism}
Let $M$ and $M'$ be closed nonorientable topological 4-manifolds with fundamental group of order two. Write $X = M \# M'$, and denote $S$ as the 3-sphere defining the connected sum.
Let $\vartheta \in \UNil_5^h(\Z;\Z^-,\Z^-)$.

\begin{enumerate}
\item
There exists a unique homeomorphism class $(X_\vartheta,h_\vartheta)$, consisting of a closed $\TOP$ 4-manifold $X_\vartheta$ and a tangential homotopy equivalence $h_\vartheta: X_\vartheta \to X$, such that it has splitting obstruction
\[
\Split_L(h_\vartheta;S) = \vartheta.
\]
The function which assigns $\vartheta$ to such a $(X_\vartheta, h_\vartheta)$ is a bijection.

\item
Furthermore,
\[
X_\vartheta \# 3(S^2\x S^2) \quad\text{is homeomorphic to}\quad X\# 3(S^2\x S^2).
\]
It admits a $\DIFF$ structure if and only if $X$ does.
There exists a $\TOP$ normal bordism between $h_\vartheta$ and $\Id_X$, with surgery obstruction $\vartheta \in L_5^h(\D_\infty^{-,-})$, such that it is composed of exactly six 2-handles and six 3-handles.
\end{enumerate}
\end{thm}

\begin{proof}
Recall that the forgetful map
\[
L_5^s(\D_\infty^{-,-}) \xra{\quad} L_5^h(\D_\infty^{-,-})
\]
is an isomorphism, since the Whitehead group $\Wh(\D_\infty)$ vanishes.
Then the existence and uniqueness of $(X_\vartheta, h_\vartheta)$ and its handle description follow from Theorem \ref{Thm_Real4}, using $r=d+1=3$ from Proposition \ref{PropB} and Proof \ref{ThmB}.
By \cite[Theorem 6]{CappellFree}, the following composite function is the identity on
$\UNil_5^h(\Z;\Z^-,\Z^-)$:
\[
\vartheta \longmapsto (X_\vartheta,h_\vartheta) \longmapsto \Split_L(h_\vartheta;S).
\]
In order to show that the other composite is the identity, note that two tangential homotopy equivalences $(X_\vartheta, h_\vartheta)$ and $(X_\vartheta', h_\vartheta')$ with the same splitting obstruction $\vartheta$ must be homeomorphic, by freeness of the $\UNil_5^h$ action on the
structure set $\cS_\TOP^h(X)$.
Finally, since the 4-manifolds $X_\vartheta$ and $X$ are $6$-stably homeomorphic via the $\TOP$ normal bordism between $h_\vartheta$ and $\Id_X$, we conclude that
they are in fact $3$-stably homeomorphic by Corollary \ref{Cor_PropB}.
\end{proof}

The six 2-handles are needed for map data and only three are needed to relate domains.

\begin{cor}\label{Cor_FakeBound}
The above theorem is true for $X = \RP^4 \# \RP^4$, with $\RP^4$ of $w_2$-type III. \qed
\end{cor}

\begin{rem}
We comment on a specific aspect of the topology of $X$.
Every homotopy automorphism of $\RP^4\#\RP^4$ is homotopic to a homeomorphism \cite[Lemma 1]{JK}. Then any automorphism of the group $\D_\infty$ can be realized \cite[Claim]{JK}.
The homeomorphism classes of closed topological 4-manifolds $X'$ in the (not necessarily tangential) homotopy type of $X$ has been computed in \cite[Theorem 2]{BDK}.
The classification involves the study \cite[Theorem 1]{BDK} of the effect of transposition of the bimodules $\Z^-$ and $\Z^-$ in the abelian group $\UNil_5^h(\Z;\Z^-,\Z^-)$.
As promised in the introduction, Corollary \ref{Cor_FakeBound} provides a \emph{uniform upper bound} on the number of $S^2 \x S^2$ connected-summands sufficient for \cite[Theorem 1(f)]{JK}, and on the number of 2- and 3-handles sufficient for \cite[Proof 1(f)]{JK}.
\end{rem}

\bibliographystyle{alpha}
\bibliography{CancellationTop4Mfld}

\end{document}